\documentclass[a4paper,12pt]{article}

\usepackage[pdfstartview = {FitH}, bookmarksopen]{hyperref} %%%%Verlinkungen im Dokument

\usepackage{amsfonts} % andere Schriftarten, wie mathcal oder mathbb

\usepackage[latin1]{inputenc}  %Umlaute

\usepackage{amssymb}   					% \rightrightarrows
\usepackage{mathrsfs}                    % \mathscr
\usepackage{amsmath}                    % Klammern, texte in Formeln
\usepackage{amsthm} 
\usepackage{xcolor}                    %Farbige Schrift etc. 
\usepackage[arrow, matrix, curve]{xy} 

\usepackage{color}

\usepackage{pdfpages}
\usepackage{multirow}   % für die Zusammenfassung von Spalten in Tabellen wir beim Schema für den Schlüsselaustausch mit semidirekten Produkten 

\usepackage{rotating}
\usepackage{rotfloat}
\usepackage{lscape} %Querseite

\usepackage{tikz}
\usetikzlibrary{shapes,arrows,calc}
\tikzstyle{decision} = [chamfered rectangle, draw, text centered, chamfered rectangle xsep=2em]
\tikzstyle{block}    = [rectangle, draw, text centered]
\tikzstyle{schrift}  = [rectangle, text centered]
\tikzstyle{line}     = [draw, -latex']  
\tikzstyle{symbol}   = [rectangle, draw, text centered, minimum height= 2em, minimum width = 2em]
\tikzstyle{syndrom}  = [rectangle, draw, text centered, minimum height= 2em, minimum width = 4em]

 \newtheorem{thm}{Theorem}[section]
 \newtheorem{cor}[thm]{Corollary}
 \newtheorem{lem}[thm]{Lemma}
 
 \theoremstyle{definition}
 \newtheorem{defn}[thm]{Definition}
 \theoremstyle{remark}
 \newtheorem{rem}[thm]{Remark}
 \newtheorem*{ex}{Example}
 \numberwithin{equation}{section}

\title{On the generalized Hurwitz equations and
the Baragar-Umeda equations}

\author{Benjamin Fine, Gabriele Kern-Isberner, \\
Anja I. S. Moldenhauer and Gerhard Rosenberger}
\date{}  %%% kein Datum

\begin{document}

\maketitle
\begin{center}\textit{In memoriam Murray Macbeath}\end{center}

\vspace{0.5cm}
\begin{abstract}
We consider the generalized Hurwitz equation $a_1x_1^2+ \cdots +a_nx_n^2 = dx_1 \cdots x_n-k$
 and the Baragar-Umeda equation $ax^2+by^2+cz^2=dxyz+e$ for solvability in integers.
\end{abstract}



%%%%%%%%%%%%%%%%%%%%%%%%%%%%%%%%%%%%%%%%%%%%%
%%%%%%%%%%%%%% Section 1 %%%%%%%%%%%%%%%%%%%%
%%%%%%%%%%%%%%%%%%%%%%%%%%%%%%%%%%%%%%%%%%%%%

\section{Introduction}

A  generalized Hurwitz equation is given by
\begin{align} \label{GH}
a_1x_1^2+\cdots +a_nx_n^2 = dx_1\cdots x_n -k
\end{align}
with $n \geq 3$, $k \in \mathbb N \cup \{0\}$, $a_1, \ldots, a_n, d \in \mathbb N$ such that $a_i | d$ for $i=1,\ldots,n$ and $gcd(a_i,a_j)=1$ for $i \ne j$.\\
Hurwitz \cite{H} itself considered the special case $a_1=\cdots=a_n=1$. In this case we call \eqref{GH} just a Hurwitz equation.\\
For $n=3$ the equation \eqref{GH} is often called the Markoff-Rosenberger equation (MR) and is quite well understood. Such an equation (MR) occurs in connection with different mathematical theories and problems; for instance with the minimum of binary quadratic forms, the Markoff spectrum and diophantine approximation (\cite{Ba5}, \cite{BLS}, \cite{LS}, \cite{M}, \cite{R2}, \cite{Sch}); with simple closed geodesics on certain Riemann and Fricke surfaces (\cite{BLS}, \cite{C2}, \cite{LS}, \cite{Sh2}, \cite{W}); with the classification of arithmetic hyperbolic surface bounds (\cite{BMR}); with the construction of series of noncongruence subgroups of the modular group and automorphisms groups of Riemann and Klein surfaces (\cite{R3}); with the generators of free groups of rank two (\cite{C1}, \cite{C2}, \cite{R4}) and with discreteness conditions for groups of $2 \times 2$ matrices (\cite{K-IR}, \cite{R5}). Goldman \cite{G} considers the automorphism group $\Gamma$ of the polynomial
\begin{align*}
k(x,y,z)=x^2+y^2+z^2-xyz-2
\end{align*}
and for $t \in \mathbb R$ the $\Gamma$-action on $k^{-1}(t) \cap \mathbb R^3$. Of special interest is the connection of the Markoff equation with the classification and description of the quiver algebras with three vertices (\cite{BBH}). Several results concern with the solutions of an equation of type (MR) in the integers, in the rational numbers and, more generally, in algebraic number fields (\cite{Ba1}, \cite{Ba3}, \cite{Ba4}, \cite{BR}, \cite{G-JT}, \cite{K-IR}, \cite{Mo1}, \cite{Mo2}, \cite{R2}, \cite{R5}, \cite{S1}, \cite{S2}, \cite{Sh1}) or with the asymptotic behavior and the growth of the integral solutions (\cite{Ba2}, \cite{Ba6}, \cite{BU}, \cite{Z}).
Hence, we here do not especially look at the equations of type (MR). They are considered just as a part of a general theory for the solvability of a generalized Hurwitz equation \eqref{GH} in integers.

The paper is based on notes we made in connection with the references \cite{H}, \cite{K} and \cite{R1}. We publish these now together with new results because we realized an upcoming interest in the Hurwitz equation (see for instance \cite{Ba1}, \cite{Ba3}, \cite{Ba6}, \cite{B}, \cite{He} and \cite{Hi}). Also we observed the existence of an interesting secret sharing protocol based on the Hurwitz equation.

In section \ref{sec2} we describe this secret sharing protocol and some known and new results for the special case of the Hurwitz equation.

In section \ref{sec3} we give solvability results for the generalized Hurwitz equation. We show especially
for a fixed $n \geq 3$ that there is a solution in natural numbers only for finitely many $a_1,\ldots,a_n,d$ and that then the number of fundamental solutions is finite (see Lemma \ref{L3} and Theorem \ref{T6}).

Finally, Baragar and Umeda \cite{BU} suggested to consider the diophantine equation
\begin{align}\label{BU}
ax^2+by^2+cz^2=dxyz+e
\end{align}
with $a,b,c,d,e \in \mathbb N$ such that $a|d$, $b|d$, $c|d$ and $gcd(a,b,c)=1$.
It is obvious that we here assume $gcd(a,b,c)=1$ because if $gcd(a,b,c)=t$ then $t|e$.\\
Baragar and Umeda ask for the existence of fundamental solutions\newline $(x,y,z)~\in~\mathbb N^3$ with
\begin{align*}
1 \leq x \leq \frac{d}{2a}yz, \quad
1 \leq y \leq \frac{d}{2b}xz \quad \text{and} \quad
1 \leq z \leq \frac{d}{2c}xy
\end{align*}
(see Lemma \ref{L3}).\\
They give for $e=1$ the complete list of equations \eqref{BU} which have fundamental solutions, and then they calculate these fundamental solutions.

In section \ref{sec4} we prove that necessarily $1 \leq e \leq 4$ for an equation \eqref{BU} to have a fundamental solution. We also give for $1 \leq e \leq 3$ a complete description of all the possible equations \eqref{BU} together which the classification of the solutions in natural numbers. We remark that if $e=4$ in equation \eqref{BU} then we have infinitely many fundamental solutions, which we completely describe.
%%%%%%%%%%%%%%%%%%%%%%%%%%%%%%%%%%%%%%%%%%%%%%%%%%%%%%%
%%%%%%%%%%%%%%Section 2%%%%%%%%%%%%%%%%%%%%%%%%%%%%%%%%
%%%%%%%%%%%%%%%%%%%%%%%%%%%%%%%%%%%%%%%%%%%%%%%%%%%%%%%
\section{The Hurwitz equation}\label{sec2}

In this section we first give a survey of some known results about the solvability of the Hurwitz equation in integers. It is
\begin{align}\label{H}
x_1^2+ \cdots + x_n^2 = dx_1\cdots x_n - k
\end{align}
with $k \in \mathbb N \cup \{0\}$, $d \in \mathbb N$, $n \geq 3$. At the end of this section we describe the secret sharing protocol based on the Hurwitz equation.\\

 We start with the solvability of \eqref{H} in the integers.\\
For this, first of all, we may restrict ourselves to the case $k=0$ because 
\begin{align*}
k=x^2_{n+1}+ \cdots + x^2_{n+k} \quad  \text{ with } \ x_{n+1}=\cdots =x_{n+k}=1.
\end{align*} 
Hence from now on let $k=0$, that is
\begin{align}\label{1}
x_1^2+ \cdots + x_n^2 = dx_1 \cdots x_n
\end{align}
with $n \geq 3$, $d \in \mathbb N$.\\
Since $d \in \mathbb N$ and $x_1^2+ \cdots + x_n^2 \geq 0$ we may restrict ourselves to the solvability of \eqref{1} in natural numbers.\\
Let 
\begin{align*}
L_n:=L_{n,d}:=\{ (x_1,\ldots, x_n) \in \mathbb N^n | (x_1,\ldots, x_n) \text{ is a solution of \eqref{1}} \}.
\end{align*}
The set $L_n$ can be empty, for instance if $n=3$, $d=1$ or $n=8$, $d=5$ (see~Theorem \ref{T2}).
In what follows we assume that $L_n \ne \emptyset$. Also we often write $\overline{x}$ for $(x_1,x_2,\ldots,x_n)$. \\
The following maps are permutations from $L_n$ onto $L_n$:
\begin{align*}
\varphi : (x_1,\ldots, x_n) &\mapsto (x_2,x_1,x_3,\ldots, x_n),\\
\omega : (x_1,\ldots, x_n) &\mapsto (x_n,x_1,\ldots, x_{n-1}) \text{ and }\\
\psi : (x_1,\ldots, x_n) &\mapsto (x'_1,x_2,\ldots, x_n)\text{ with } x'_1=dx_2\cdots x_n - x_1.
\end{align*}
Let $M_n:=M_{d,n}:=$ $< \varphi, \omega, \psi >$ be the permutation group generated by $\varphi, \omega$ and 
$\psi$. We remark that $< \varphi, \omega >$ $=S_n$, with $S_n$ the full permutation group of the $n$ symbols $x_1,\ldots,x_n$.

\begin{thm}\label{T1}\cite{R1}.~\\
If $L_n \ne \emptyset$ then $M_n$ can be described by the generators $\varphi, \omega, \psi$ and the following defining relations:
\begin{enumerate}
\item $\varphi^2 = \omega ^n = (\varphi \omega)^{n-1} = (\varphi \omega^2 \varphi \omega^{-2})^2 = (\varphi \omega \varphi \omega^{-1})^{3} = \psi ^2 = (\psi \omega \varphi \omega^{-1})^2=$ \newline $= \psi \varphi \omega \psi \omega^{-1}\varphi = 1$ \  if $n \geq 4$ and 
\item $\varphi^2 = \omega^3 =(\varphi \omega)^2 = \psi^2 = (\psi \varphi \omega)^2 =1$ \ if $n=3$.
\end{enumerate}
\end{thm}

\begin{cor}\label{C1}\cite{R1}.
\begin{enumerate}
\item If $3 \leq m \leq n$ then $M_m$ is a subgroup of $M_n$.
\item $M_3 \cong PGL(2, \mathbb Z)$ and $M_3^{0} \cong PSL(2,\mathbb Z)$, where $M_3^{0}$ is generated by the permutations
\begin{align*}
\rho=\psi \varphi \omega:& (x_1,x_2,x_3) \mapsto (dx_2x_3 - x_1, x_3,x_2) \qquad \text{and}\\
\omega:& (x_1,x_2,x_3) \mapsto (x_3,x_1,x_2).
\end{align*}
\end{enumerate}
\end{cor}

\begin{lem}\label{L1}\cite{H}, \cite{R1}.~\\
Let $L_n \ne \emptyset$ and $\overline{x}=(x_1,\ldots,x_n)\in L_n$.\newline Then there exists an $\overline{y}=(y_1,\ldots,y_n)\in L_n$ with $$1\leq y_1\leq y_2 \leq \cdots \leq y_n \leq \frac{d}{2}y_1y_2 \cdots y_{n-1}$$ and a permutation $\gamma \in M_n$ with $\gamma(\overline{y})=\overline{x}$.\\
If there exists another $\overline{z}=(z_1,\ldots,z_n) \in L_n$ with $$1\leq z_1\leq z_2 \leq \cdots \leq z_n \leq \frac{d}{2}z_1z_2 \cdots z_{n-1}$$ and some $\beta \in M_n$ with $\beta(\overline{y})= \overline{z}$ then $y_i = z_i$ for $i=1,\ldots,n$.
\end{lem}

\begin{defn}
We call an $\overline{y}=(y_1,\ldots, y_n) \in L_n$ with $$1\leq y_1\leq y_2 \leq \cdots \leq y_n \leq \frac{d}{2}y_1y_2 \cdots y_{n-1}$$ a fundamental solution of \eqref{1} for $d$.
\end{defn}

\begin{lem}\label{L2}\cite{H}~\\
If $L_n \ne \emptyset$ then for each $d \in \mathbb N$ there are only finitely many fundamental solutions of \eqref{1}.
\end{lem}

\begin{rem}
 If we want to find all solutions $\overline{x} \in L_n$ then it is enough to find all fundamental solutions of \eqref{1}. Let $F_n$ be the number of fundamental solutions of \eqref{1}. It is possible that $|F_n| \geq 2$ for $d$ if $L_n \ne \emptyset$.
\end{rem}

\begin{ex}
For $n=95$ and $d=1$ there are the fundamental solutions
\begin{align*}
(1, \ldots, 1, 3, 9, 13) \qquad  \text{ and } \qquad (1, \ldots, 1, 4, 6, 12).
\end{align*}
This cannot happen if $3 \leq n \leq 2d$.
\end{ex}

\begin{thm}\label{T2}\cite{R1}\\
If $3 \leq n \leq 2d$ then $L_n \ne \emptyset$ if and only if one of the following cases holds:
\begin{enumerate}
\item $n=d$;
\item $n=6$, $d=3$;
\item $n=7,10,13,16$ and $d=5,6,7,8$, respectively;
\item $n \equiv 1  \pmod 2$, $d=\frac{n+3}{2}$.
\end{enumerate}
In all these cases the group $M_n$ operates transitively on $L_n$, that is, in all these cases there exists exactly one fundamental solution of \eqref{1} for $d$, and these are given as follows:

	\begin{table}[h]
	%\begin{small}
	\begin{center}
	%\setlength{\tabcolsep}{.3\tabcolsep}
	%\caption{Nielsen transformations (NTs) of the dealer I}\label{vorgehen}
	\begin{tabular}{c | c | c| c | c }
	$n$ &$d$ & $x_i$ \text{ for } $i=1,\ldots,n-2$  &$ x_{n-1}$ & $x_n$\\
	\hline
	$n$ & $n$ & $1$ & $1$ & $1$\\
	$6$ & $3$ & $1$ & $2$ & $2$\\
	$7$ & $5$ & $1$ & $1$ & $2$\\
	$10$ & $6$ & $1$ & $1$ & $3$\\
	$13$ & $7$ & $1$ & $1$ & $3$\\
	$16$ & $8$ & $1$ & $1$ & $3$\\
	$5 \leq n$, $n \equiv 1 \pmod2$& $\frac{n+3}{2}$& $1$ & $1$ & $2$ \\
	\end{tabular}
	\end{center}
	%\end{small}
	\end{table}
\end{thm}

\begin{rem}
Some more numerical results are stated in \cite{B} and \cite{He}.
\end{rem}

Herzberg \cite{He} gave an efficient algorithm to find pairs $(d,n)$ with $d<n$ for which \eqref{1} has nontrivial solutions in $L_n$.
He was the first who published examples with $|F_n|\geq 2$, in fact if $d=1$ then $|F_{14}|=2$ and $|F_{19}|=3$.
Baragar~\cite{Ba1} described the frequency such that $F_n \ne \emptyset$ for any fixed $n$.

\begin{defn}
For $n \in \mathbb N$, $n \geq 3$, let
\begin{align*}
A(n)= \# \{d \in \mathbb N | F_n \ne \emptyset\}.
\end{align*}
\end{defn}

\begin{thm}\cite{Ba1}~\\
For every $\epsilon >0$,
\begin{align*}
A(n) = O(n^{\frac{1}{2}+\epsilon}).
\end{align*}
\end{thm}
%%%%%%%%%%%%%%%%%%%2) S. 16%%%%%%%%%%%%%%%%%%%%
Let $R$ be a finite field or ring and $d, b \in R$ with $d \ne 0$.

\begin{rem}
\begin{enumerate}
\item Let
\begin{align*}
L_R=\{ (x_1,x_2,x_3) | x_1^2+x_2^2+x_3^2-dx_1x_2x_3=b \} \subset R^3
\end{align*}
be not empty.\\
Let 
\begin{align*}
\varphi_R: (x_1,x_2,x_3) &\mapsto (x_2,x_1,x_3),\\
\omega_R: (x_1,x_2,x_3) &\mapsto (x_3,x_1,x_2),\\
\psi_R: (x_1,x_2,x_3) &\mapsto (dx_2x_3-x_1,x_2,x_3) \ \text{ and }\\
\rho_R: (x_1,x_2,x_3) &\mapsto (dx_2x_3-x_1,x_3,x_2)
\end{align*}
and
\begin{align*}
M_R = \ < \varphi_R, \omega_R, \psi_R >, \qquad M_R^0 = \ <\omega_R, \rho_R >.
\end{align*}
\begin{enumerate}
\item[(a)] Using this we may construct finite images of the modular group \newline $PSL(2,\mathbb Z)~\cong~M_3^0$ and the extended modular group\newline  $PGL(2,\mathbb Z) \cong M_3$.\\
We made several computational experiments and found some finite simple groups as some alternating groups $A_n$, $n \geq 5$, some projective linear groups $PSL(2,\mathbb  K)$, $\mathbb K$ a finite field, and the sporadic group $M_{12}$, the Mathieu group $M_{12}$. For a detailed discussion of finite simple groups which are images of the modular group see \cite{Wo1} and \cite{Wo2}. Similar computational experiments together with some explicit results for finite fields are given by Holt and Macbeath \cite{HM}.
\item[(b)] We also found some maximal automorphism groups of compact Klein surfaces with nonempty boundary and genus $g \geq 2$\newline ($M^*$~groups). These have order $12(g-1)$ and are finite images of the extended modular group $PGL(2,\mathbb Z)$.
\end{enumerate}
\item Goldman \cite{G} considers the whole group $\Gamma$ of automorphisms of the polynomial
\begin{align*}
k(x,y,z) = x^2+y^2+z^2 -xyz - 2.
\end{align*}
From above it is clear that - up to isomorphisms - the group $PGL(2,\mathbb Z)$ is a subgroup of $\Gamma$. The group $\Gamma$ is generated by the $PGL(2,\mathbb Z)$ and the mappings of order $2$ which replace two of the $x,y,z$ by their negative values.\\
For $t \in \mathbb R$, he describes in detail the $\Gamma$-action on $k^{-1}(t)\cap \mathbb R^3$. For parts of his results he uses the iterative procedure developed in the proof of the well known Lemma 2 of \cite{K-IR}.
He claims that the proof contains a gap near the end and gives a slightly different version of the iterative procedure. Goldman does not say where he believes to see a gap. We want to mention that the proof of Lemma 2 of \cite{K-IR} definitely does not contain a gap. Maybe Goldman overlooked the reference to the paper \cite{KR} for more details concerning the iterative procedure (see also \cite{R4} and especially \cite{FR} where we classified all generating pairs of all two generator Fuchsian groups) and the fact that Lemma 2 is trivial if $0 \leq x < 2$, which we see from the procedure and which is explained near the end of the proof of Lemma 2.
If we assume that the iterative procedure does not lead to the statement of Lemma 2, then for the limit elements $x_0, y_0, z_0$ we must have $ 0 \leq x_0 \leq y_0 \leq z_0 \leq \frac{1}{2}x_0y_0$. This gives $2\leq x_0$ from $z_0 \leq \frac{1}{2}x_0y_0$. If $x_0>2$ then from the inequalities and the calculated expressions for $z_0$ we would have $x_0^2(x_0-2) \leq y_0^2 (x_0 - 2) \leq x_0^2 - c < x_0^2 -4$ because $c > 4$, and then $x_0^2 < x_0 + 2$ which is impossible for $x_0 > 2$. Therefore $x_0 = 2$ and $y_0 = z_0$ as claimed near the end of the procedure, and this contradicts $c>4$.
\end{enumerate}
\end{rem}

Now we describe the announced $(n,t)$ secret sharing protocol based on Hurwitz equation \eqref{H}, which is
\begin{align*}
x_1^2+x_2^2+ \cdots + x_m^2 = x_1 \cdots x_m -k.
\end{align*}
We consider this equation over a field $\mathbb K$, for example $\mathbb K = \mathbb Q$ or a big finite field, with $k \ne 0$.

An $(n,t)$ secret sharing protocol, with $n,t \in \mathbb N$ and $t \leq n$, is a method to distribute a secret $S$ among a group of $n$ participants in such a way that it can be recovered if at least $t$ of them combine their shares.\\ 
The secret in this protocol is the element 
\begin{align*}
S:= x_1^2+x_2^2+ \cdots + x_m^2 - x_1 \cdots x_m.
\end{align*}
The shares for the participants are subsets from $\{x_1,x_2, \ldots, x_m\}$. To generate these shares we use the method from D. Panagopoulos (see \cite{P}):

\begin{enumerate}
\item It is  $m=\binom{n}{t-1}$ the number of  elements the participants need to know to reconstruct the secret, that is, they have to know the set $\{x_1,x_2,\ldots,x_m\}$.
\item Let $A_1,A_2,\ldots,A_m$ be an enumeration of the subsets of $\{1,2,\ldots,n\}$ with $t-1$ elements. Define $n$ subsets $R_1,R_2,\ldots,R_n$ of $\{x_1,x_2,\ldots,x_m\}$ with the property
$$
      x_j \in R_i  \qquad \Longleftrightarrow \qquad i \not\in A_j
$$
for $j=1, 2, \ldots, m$ and  $i=1, 2, \ldots, n$.
     
\item Each of the $n$ participants gets one of the sets $R_1,R_2,\ldots,R_n$.
\end{enumerate}
Each element $x_j$ is exactly contained in $n-(t-1)$ subsets. Hence for each\newline $j=1,2,\ldots,m$ the element $x_j$ is not contained in $t-1$ subsets from\newline $\{R_1,R_2,\ldots,R_n\}$. As a consequence, $x_j$ is in each union of $t$ subsets. On the other hand if just $t-1$ arbitrary sets from $\{R_1,R_2, \ldots, R_n\}$ are combined, there exist  a $j$ so that the element $x_j$ is not included in the union of this sets.

If just one element $x_j$ is absent the participants do not get the element $S$ and hence cannot compute the secret.

If $t$ of $n$ participants come together they get by construction the set $\{x_1,x_2,\ldots,x_m\}$ and hence they can calculate the secret
\begin{align*}
S= x_1^2+x_2^2+ \cdots + x_m^2 - x_1 \cdots x_m.
\end{align*}

%%%%%%%%%%%%%%%%%%%%%%%%%%%%%%%%%%%%%%%%%%%%%%%%%%%%%%%
%%%%%%%%%%%%%%Section 3%%%%%%%%%%%%%%%%%%%%%%%%%%%%%%%%
%%%%%%%%%%%%%%%%%%%%%%%%%%%%%%%%%%%%%%%%%%%%%%%%%%%%%%%
\section{The generalized Hurwitz equation}\label{sec3}

In this section we consider the diophantine equation \eqref{GH}, which is
\begin{align*}
a_1x_1^2+\cdots + a_nx_n^2 = dx_1\cdots x_n - k
\end{align*}with $n\geq3$, $k \in \mathbb N \cup \{0\}$, $a_1, \ldots, a_n, d \in \mathbb N$, $a_i | d$ for $i=1,\ldots, n$ in \eqref{2} and $gcd(a_i,a_j)=1$ for $i \ne j$, for solvability in the integers.\\
As in section \ref{sec2} we may restrict ourselves to the case $k=0$ because 
\begin{align*}
k=a_{n+1}x^2_{n+1}+ \cdots +a_{n+k}x^2_{n+k} \quad  \text{ with } \\ a_{n+1}=\cdots = a_{n+k}=x_{n+1}=\cdots =x_{n+k}=1.
\end{align*}
Hence, from now on let $k=0$, that is,
\begin{align}\label{2}
a_1x_1^2+\cdots +a_nx_n^2 = dx_1\cdots x_n
\end{align}
with $n\geq3$, $a_1,\ldots,a_n, d \in \mathbb N$, $a_i | d$ for $i=1,\ldots, n$ and $gcd(a_i,a_j)=1$ for $i~\ne~j$.
Since $a_1,\ldots,a_n,d \in \mathbb N$ and $a_1x_1^2+\cdots +a_nx_n^2 \geq 0$ we may restrict ourselves to the solvability of \eqref{2} in natural numbers.
The assumption $gcd(a_i,a_j)=1$ for $i \ne j$ is not a restriction for $n=3$ because if, for instance, $t|a$ and $t|b$ then also $t|c$. But it is certainly a restriction for $n\geq 4$.\\
Again, let
\begin{align*}
L_n:=L_{a_1,\ldots,a_n,d,n}:=\{ (x_1,\ldots, x_n) \in \mathbb N^n | (x_1,\ldots,x_n) \text{ is a solution of \eqref{2}}\}.
\end{align*}
As in section \ref{sec2}, $L_n$ can be empty, and in what follows we assume $L_n \ne \emptyset$. Also we often write $\overline{x}$ for $(x_1,\ldots, x_n)$. The following maps are permutations from $L_n$ onto $L_n$:
\begin{align*}
\psi_i: (x_1,\ldots, x_{i-1},x_i,x_{i+1},\ldots, x_n) \mapsto (x_1,\ldots, x_{i-1},x'_i,x_{i+1},\ldots, x_n) 
\end{align*}
for $i=1,\ldots, n$ with 
\begin{align*}
x'_i = \frac{d}{a_i} \prod_{\substack{j=1\\ j\ne i}}^{n} x_j - x_i.
\end{align*}
We remark that $x'_i \in \mathbb N$ because $a_i | d$ for $i=1,\ldots,n$.\\
Let
\begin{align*}
M_n:=M_{a_1,\ldots,a_n,d,n}= \ < \psi_1,\ldots,\psi_n>
\end{align*}
be the permutation group generated by $\psi_1,\ldots,\psi_n$.

\begin{thm}
If $L_n \ne \emptyset$ then $M_n$ can be described by the generators $\psi_1,\ldots,\psi_n$ and the following defining relations
\begin{align*}
\psi_1^2=\psi_2^2=\cdots=\psi_n^2=1,
\end{align*}
that is, $M_n$ is the free product of $n$ cyclic groups of order $2$.
\end{thm}

\begin{proof}
Certainly $\psi_i^2=1$ for $i=1,\ldots,n$. We write $y_i=\sqrt{a_i}x_i \in \mathbb R$ if $(x_1,\ldots,x_n)\in L_n$. Then we get the equation 
\begin{align*}
y_1^2+\cdots +y_n^2 = \frac{d}{\sqrt{a_1\cdots a_n}}  y_1\cdots y_n.
\end{align*}
Now we are exactly in the situation of Satz 1 in \cite{R1} where we considered the equation
\begin{align}\label{3}
y_1^2+\cdots +y_n^2 = a y_1\cdots y_n
\end{align}
with $a \in \mathbb R$, $a>0$, and worked with real solutions $(y_1,\ldots,y_n)\in \mathbb R^n$.\\
Hence, we get $\psi_{r_1}\cdots \psi_{r_m} \ne 1$ if $r_i \ne r_{i+1}$ for all $i=1,\ldots,m-1$.
This gives the result.
\end{proof}

\begin{defn}
Let $\overline{x}=(x_1,\ldots,x_n) \in \mathbb N^n$ be a solution of \eqref{2}.
\begin{enumerate}
\item $h(\overline{x})=x_1+\cdots+x_n$ is called the height of $\overline{x}$.
\item The solutions $\psi_i(\overline{x})$ are called the neighbors of $\overline{x}$.
\end{enumerate}
\end{defn}

This means, that $\overline{x}$ has exactly the $n$ neighbors  $\psi_1(\overline{x}),\ldots,\psi_n(\overline{x})$.

\begin{defn}
A solution $\overline{x}=(x_1,\ldots,x_n)\in L_n$ is called a fundamental solution of \eqref{2} if 
\begin{align*}
h(\psi_i(\overline{x})) \geq h(\overline{x}) \qquad \text{ for all } i=1, \ldots, n.
\end{align*}
\end{defn}

\begin{lem}\label{L3}
The $n$-tuple $\overline{x}=(x_1,\ldots,x_n)\in L_n$ is a fundamental solution of \eqref{2} if and only if 
\begin{align*}
2a_ix_i \leq d \prod_{\substack{j=1\\ j\ne i}}^{n} x_j \qquad \text{ for all } i=1,\ldots, n.
\end{align*}
\end{lem}

\begin{proof}
If $\overline{x}$ is a fundamental solution of \eqref{2} then
\begin{align*}
x_i \leq \frac{d}{a_i} \prod_{\substack{j=1\\ j\ne i}}^{n} x_j - x_i \qquad \text{ for all } i=1,\ldots, n
\end{align*}
which is equivalent to 
\begin{align*}
2a_ix_i \leq d \prod_{\substack{j=1\\ j\ne i}}^{n} x_j \qquad \text{ for all } i=1,\ldots, n.
\end{align*}
If these inequalities hold then $h(\overline{x}) \leq h(\psi_i(\overline{x}))$ for all $i=1,\ldots, n$, that is, $\overline{x}$ is a fundamental solution of \eqref{2}.
\end{proof}

\begin{thm}\label{T5}
\begin{enumerate}
\item[(1)] If $\overline{x} \in L_n$ then there exists a $\gamma \in M_n$ with $\gamma(\overline{x})$ is a fundamental solution.
\item[(2)] If $\overline{x},\overline{y} \in L_n$ are two different fundamental solutions of \eqref{2} then there is no $\gamma \in M_n$ with $\gamma(\overline{x})=\overline{y}$.
\end{enumerate}
\end{thm}

\begin{proof}
(1) is obviously because $\overline{x} \in \mathbb N^n$.\\
 We now prove (2). Let $\overline{x},\overline{y} \in L_n$ be two different fundamental solutions of \eqref{2}. Assume there exists a $\gamma \in M_n$ with $\gamma(\overline{x})=\overline{y}$, that is, there is a finite sequence $\overline{x}_0,\overline{x}_1,\ldots,\overline{x}_p$ with $\overline{x}_0=\overline{x}$, $\overline{x}_p=\overline{y}$ and $\overline{x}_0,\overline{x}_1,\ldots,\overline{x}_p$ are pairwise different and $\overline{x}_i$ is a neighbor of $\overline{x}_{i+1}$ for $i=0,1,\ldots,p-1$.
We have $p \geq 2$ because $\overline{x}$ and $\overline{y}$ are  different and the $\overline{x}_i$ are pairwise different. 
Especially $h(\overline{x}_i)\ne h(\overline{x}_{i+1})$ for $i=0,1,\ldots,p-1$, because otherwise $\overline{x}_i = \overline{x}_{i+1}$.
Also $h(\overline{x}_1) > h(\overline{x}_0)$, $h(\overline{x}_{p-1}) > h(\overline{x}_p)$ because $\overline{x}_0$ and $\overline{x}_p$ are a fundamental solution of \eqref{2}. Let
\begin{align*}
h^* = \underset{0 \leq i \leq p}{max} h(\overline{x}_i) = h(\overline{x}_s)
\end{align*}
for $ 0 \leq s \leq p$.
We remark that $s \ne 0$ and $s\ne p$, so $0<s<p$.
The solutions $\overline{x}_{s-1}$ and $\overline{x}_{s+1}$ are neighbors of $\overline{x}_s$. Without loss of generality, let $\overline{x}_{s-1}= \psi_1(\overline{x}_s)$, $\overline{x}_{s+1}= \psi_2(\overline{x}_s)$, that is
\begin{align*}
\overline{x}_{s-1} &= (\frac{d}{a_1}x_{2_s}\cdots x_{n_s} - x_{1_s},x_{2_s},\ldots, x_{n_s}),\\
\overline{x}_{s+1} &= (x_{1_s},\frac{d}{a_2}x_{1_s}x_{3_s}\cdots x_{n_s} - x_{2_s},x_{3_s},\ldots, x_{n_s})
\end{align*}
when we write $\overline{x}_s=(x_{1_s},\ldots, x_{n_s})$.\\
From 
\begin{align*}
h(\overline{x}_s) \geq h(\overline{x}_{s-1})\quad \text{ and } \quad h(\overline{x}_s) \geq h(\overline{x}_{s+1})
\end{align*}
 we get
\begin{align*}
x_{1_s} \geq \frac{d}{a_1}x_{2_s}\cdots x_{n_s} - x_{1_s} \quad \text{ and } \quad x_{2_s} \geq \frac{d}{a_2}x_{1_s}x_{3_s}\cdots x_{n_s} - x_{2_s}
\end{align*}
which is equivalent to
\begin{align*} 
2a_1x_{1_s}^2 \geq dx_{1_s}x_{2_s}\cdots x_{n_s} \quad \text{ and } \quad 2a_2x_{2_s}^2 \geq dx_{1_s}x_{2_s}\cdots x_{n_s}.
\end{align*}
Hence we get 
\begin{align*}
a_1x_{1_s}^2+a_2x_{2_s}^2 \geq dx_{1_s}\cdots x_{n_s} = a_1x_{1_s}^2+\cdots +a_nx_{n_s}^2
\end{align*}
which is impossible because $n \geq 3$ and $x_{1_s},\ldots, x_{n_s} >0$.\\
This proves Theorem \ref{T5}.
\end{proof}

\begin{rem}
We call two element $\overline{x},\overline{y} \in L_n$ \textbf{equivalent} if there is a $\gamma \in M_n$ with $\overline{x}=\gamma(\overline{y})$.\\
Again, let $F_n$ be the set of the fundamental solutions of \eqref{2}. Theorem \ref{T5} then means that $M_n$ operates discontinuously on $L_n$, and $F_n$ is a fundamental domain for this operation.\\
Also, we get all solutions of \eqref{2} if we know all fundamental solutions of \eqref{2}.
\end{rem}

\begin{lem}\label{L4}
Let $\overline{x}=(x_1,\ldots,x_n) \in F_n$ be a fundamental solution of \eqref{2} with $1 \leq x_1 \leq x_2 \leq \cdots \leq x_n$. Then
 $$dx_1\cdots x_{n-2} \leq a_1 + a_2 + \cdots + a_n$$ 
and equality holds if and only if $x_1=x_2=\cdots=x_n$.
\end{lem}

\begin{proof}
Since $\overline{x} \in F_n$ we have 
\begin{align*}0 < x_1 \leq \cdots \leq x_n \leq \frac{d}{2a_n}x_1\cdots x_{n-1}.
\end{align*}
The function
\begin{align*}
f(t) = a_1x_1^2 + \cdots +a_{n-1}x_{n-1}^2+a_nt^2 - dx_1\cdots x_{n-1}t
\end{align*}
decrease monotonly in $\lbrack 0, \frac{d}{2a_n}x_1\cdots x_{n-1} \rbrack $ because of
\begin{align*}
f'(t)= 2a_nt - dx_1\cdots x_{n-1} \leq 0 \quad \text{ if } \quad t \leq \frac{d}{2a_n}x_1\cdots x_{n-1},
\end{align*}
that is especially 
\begin{align*}
f(x_{n-1}) \geq f(x_n) = 0.
\end{align*}
Hence, 
\begin{align*}
&(a_1+\cdots + a_{n-1})x_{n-1}^2+a_nx_{n-1}^2 - dx_1\cdots x_{n-2}x_{n-1}^2 \geq\\
&\geq a_1x_1^2 + \cdots + a_{n-1}x_{n-1}^2+a_nx_{n-1}^2 - dx_1\cdots x_{n-2}x_{n-1}^2 = f(x_{n-1}) \geq 0.
\end{align*}
Therefore,
$$
dx_1 \cdots x_{n-2} \leq a_1 + a_2 + \dots +a_n,
$$
and equality holds if and only if $x_1=\cdots= x_n$ by the above inequalities.
\end{proof}

\begin{cor}\label{C2}
If $d > a_1 + \cdots + a_n$ then $L_n = \emptyset$.
\end{cor}

\begin{cor}\label{C3}
If $L_n \ne \emptyset$ then $d \leq a_1 + \cdots + a_n$ and if $d=a_1+\cdots + a_n$ then $x_1= \cdots = x_n=1$.
\end{cor}

There is a sharper version of Lemma \ref{L4}.

\begin{lem}\label{L5}
Let $\overline{x}=(x_1,\ldots,x_n) \in F_n$ and let $z_1,\ldots, z_{n-2}$ be the $n-2$ smallest numbers of $x_1,\ldots, x_n$. Then 
\begin{align*}
dz_1\cdots z_{n-2} \leq a_1 + a_2 + \cdots +a_n,
\end{align*}
and equality holds if and only if $x_1=x_2=\cdots=x_n$.
\end{lem}

\begin{proof}
We assume that $x_i = z_i$ for $i=1,\ldots,n-2$ and $x_1\leq x_2\leq \cdots \leq x_{n-2}$ (recall that we have no assumption on the size of the $a_i$).\\
 Then $x_{n-2} \leq x_{n-1}, x_n$.
If $x_{n-1} \leq x_n$ then we are in the situation of Lemma~\ref{L4}. If $x_n < x_{n-1}$ than we replace $(x_{n-1}, x_n)$ by $(x_n,x_{n-1})$. 
\end{proof}

\begin{thm}\label{T6}
Let $n \geq 3$ be fixed. Then $L_n \ne \emptyset$ only for finitely many $a_1,\ldots, a_n, d$. Especially $F_n$ is finite if $L_n \ne \emptyset$.
\end{thm}

\begin{proof}
Since $a_i|d$, $1 \leq i \leq n$, and $gcd(a_i,a_j)=1$ for $i \ne j$ we have $d~=~ga_1\cdots a_n$ for some $g \in \mathbb N$.
Hence, from Corollary \ref{C3}, we get $$a_1\cdots a_n \leq a_1+ \cdots + a_n$$ and therefor
\begin{align}\label{*}
1 \leq \frac{a_1 + \cdots + a_n}{a_1\cdots a_n}.
\end{align} 
We remark that the function 
\begin{align*}
\varphi (y_1,\ldots, y_n):= \frac{y_1+ \cdots + y_n}{y_1 \cdots y_n}, \quad y_1,\ldots,y_n > 0,
\end{align*}
decreases monotonly for each component $y_i$ because
\begin{align*}
  &\frac{\partial \varphi}{\partial y_i}
 = - \frac{1}{y_i^2} \frac{\sum_{i \ne j} y_i}{\prod_{i \ne j} y_i} < 0.
\end{align*}
We consider the following two cases:\\
\underline{Case \phantom{I}I:} At least one $a_i$ is bigger than $n$.\\
\underline{Case II:} It is $a_i \leq n$ for all $i=1,\ldots,n$.\\

\hspace{-0.8cm}\underline{Case I:} \\
Without loss of generality, let $a_n= m >n$.\\
Then necessarily $a_1=\cdots = a_{n-1}=1$ because otherwise
\begin{align*}
\frac{a_1+ \cdots +a_{n-1}+m}{a_1\cdots a_{n-1}m} \leq \frac{m+2+(n-2)}{2m} = \frac{m+n}{2m} < \frac{2m}{2m} =1
\end{align*}
by the monotony of the function $\varphi$, and this contradicts \eqref{*}.\\
Hence $a_1=\cdots=a_{n-1}=1$. This leads to
\begin{align*}
d=gm \leq (n-1) + m < 2m
\end{align*}
and therefore $g=1$, that is, d=m. But then \eqref{2} becomes
\begin{align*}
x_1^2+\cdots+ x_{n-1}^2 + mx_n^2 = mx_1\cdots x_n.
\end{align*}
Let $\overline{x}=(x_1,\ldots,x_n) \in F_n$ be a fundamental solution of \eqref{2}. Without loss of generality, we may assume $$x_1 \leq x_2 \leq \ldots \leq x_{n-1}.$$
Let $z_1,\ldots,z_{n-2}$ be the $n-2$ smallest of the $x_1,\ldots,x_n$.\\
 Then $mz_1 \cdots z_{n-2} \leq n-1+m < 2m$ by Lemma \ref{L5}. But this leads to $$z_1=\cdots=z_{n-2}=1.$$
We show that $x_n=1$.\\
Assume $x_n>1$ then necessarily $x_1=\cdots=x_{n-2}=1$, and hence,  $$(n-2)+x_{n-1}^2 + mx_n^2=mx_{n-1}x_n.$$ Since $\overline{x} \in F_n$ we get
\begin{align}\label{5}
2x_n \leq x_{n-1}= \frac{mx_n}{2} - \frac{1}{2} \sqrt{m^2x_n^2 - 4mx_n^2 - 4(n-2)}.
\end{align}
Hence,
\begin{align*}
m^2x_n^2 - 4mx^2_n - 4(n-2) \leq (mx_n-4x_n)^2 = m^2x_n^2 - 8mx_n^2 + 16x_n^2.
\end{align*}
It follow that
\begin{align*}
4mx_n^2 - 16x_n^2 \leq 4(n-2) \quad \text{ and hence } \quad x_n^2(m-4) \leq n-2.
\end{align*}
This gives $m=4$ or $m\ne 4$ and $1 \leq x_n^2 \leq \frac{n-2}{m-4} < 4$. In this second case we have $x_n=1$ which contradicts our assumption $x_n>1$.
Therefore $m=4$. But then $$m^2x_n^2-4mx_n^2-4(n-2)=-4(n-2)$$ which contradicts with \eqref{5} that $x_{n-1} \in \mathbb N$.\\
Hence, altogether $x_n=1$.\\
 But then $x_1= \cdots =x_{n-3}=1$ because $x_1 \leq x_2 \leq \cdots \leq x_{n-1}$, and we get the equation 
\begin{align*}
(n-3)+x_{n-2}^2+x_{n-1}^2+m=mx_{n-2}x_{n-1}.
\end{align*}
Again, since $\overline{x} \in F_n$ and $x_{n-2} \leq x_{n-1}$ we get
\begin{align*}
x_{n-2} \leq x_{n-1} = \frac{mx_{n-2}}{2} - \frac{1}{2} \sqrt{m^2x_{n-2}^2 - 4x_{n-2}^2 - 4m - 4(n-3)}
\end{align*}
that is,
\begin{align*}
x_{n-2}^2(m-2) \leq m+n-3 \leq 2m-4 = 2(m-2).
\end{align*}
Hence $x_{n-2}^2 \leq 2$, and then $x_{n-2}=1$. Now we get the equation
\begin{align*}
mx_{n-1} - m = n-2+x_{n-1}^2.
\end{align*}
Since $\overline{x} \in F_n$ and $x_n=1$ we have with Lemma \ref{L3} the inequality $x_{n-1} \leq \frac{m}{2}$. Therefore $\frac{m}{2}x_{n-1}-m \leq n-2$, that is, $m(\frac{x_{n-1}}{2} -1 ) \leq n-2 < m$ and therefore $x_{n-1}<4$.
We consider again the above equation
\begin{align*}
mx_{n-1}-m = n-2+x_{n-1}^2.
\end{align*}
Here $x_{n-1}= 1$ is not possible. Now, $x_{n-1}=2$ if and only if $m=n+2$. We are left with the case $x_{n-1}=3$.
Then $n+7 = 2m > 2n$, and hence $3 \leq n \leq 6$ with $n$ odd. The case $n=3$ is not possible because otherwise $m=5 < 2 x_{n-1}$ which contradicts $\overline{x} \in F_n$ (Lemma \ref{L3}). Therefore $n=5$, and $m=6$, if and only if $x_{n-1}=3$. This completes case I.\\

\hspace{-0.8cm}\underline{Case II:} \\
Now it is $a_1,\ldots,a_n \leq n$. Let $\overline{x}=(x_1,\ldots,x_n) \in F_n$.
To prove the theorem we may assume that $x_1 \leq x_2 \leq \cdots \leq x_n$ because we have no order for the $a_1,\ldots,a_n$ by size.
Then 
\begin{align*}
dx_1\cdots x_{n-2} \leq  a_1  + \cdots + a_n \leq n^2
\end{align*}
by Lemma \ref{L4}.
If $n$ is fixed then this inequality gives bounds for  the possible values for $a_1,\ldots,a_n,d,x_1,\ldots,x_{n-2}$. For each of the possible combinations of these values the numbers
\begin{align*}
c_1&:= a_1x_1^2+ \cdots + a_{n-2}x_{n-2}^2 \quad \text{ and }\\
c_2&:= dx_1\cdots x_{n-2}
\end{align*}
can be considered as constants. This leads to the equation 
\begin{align}\label{4}
a_{n-1}x_{n-1}^2+a_nx_n^2+c_1 = c_2x_{n-1}x_n
\end{align}
in the two variables $x_{n-1},x_n$.
Now, since $x_{n-1}\leq x_n$ and $\overline{x} \in F_n$, we have
\begin{align*}
x_{n-1} \leq x_n, \quad x_{n-1}\leq \frac{c_2}{2a_{n-1}}x_n \ \text{ and } \ x_n \leq \frac{c_2}{2a_n}x_{n-1}. 
\end{align*}
Then 
\begin{align*}
c_2x_{n-1}x_n - a_n x_n^2 = a_{n-1}x_{n-1}^2+c_1 \leq \frac{c_2}{2}x_{n-1}x_n + c_1
\end{align*}
and therefore
\begin{align*}
\frac{c_2}{2} x_{n-1}x_n - a_nx_n^2 \leq c_1.
\end{align*}
It follows 
\begin{align*}
x_n\left( \frac{c_2}{2}x_{n-1}-a_nx_n \right) \leq c_1.
\end{align*}
Since 
\begin{align*}
\frac{c_2}{2}x_{n-1}-a_nx_n \geq 0
\end{align*}
we get
\begin{align*}
x_n \leq 2c_1 \quad \text{ or } \quad \frac{c_2}{2} x_{n-1} - a_n x_n =0.
\end{align*}
In this first case we have bounds for $x_n$ and hence also for $x_{n-1}$.
Now let $\frac{c_2}{2}x_{n-1}-a_nx_n=0$. Then 
\begin{align*}
x_n= \frac{c_2}{2a_n} x_{n-1} \quad \text{ and } \quad x_{n-1}^2 \left( c_2^2 - 4a_na_{n-1} \right) = 4a_nc_1 >0
\end{align*}
from equation \eqref{4}.
Therefor $x_n$ and $x_{n-1}$ are uniquely determined if $L_n \ne \emptyset$.\\
This proves Theorem \ref{T6}.
\end{proof}

\begin{cor}\label{C4}
Let $n \geq 3$, $a_1,\ldots, a_n, d \in \mathbb N$, $a_i|d$ for $i=1,\ldots,n$ and $gcd(a_i,a_j)=1$ for $i \ne j$. Let $a_n > n$. Then the diophantine equation 
\begin{align*}
a_1x_1^2+ \cdots + a_n x_n^2 = d x_1 \cdots x_n
\end{align*}
has a solution $(x_1,\ldots,x_n) \in \mathbb N^n$ if and only if $a_1= \cdots =a_{n-1}=1$,\\
$d=a_n > n$ and one of the following cases holds:
\begin{enumerate}
\item[(1)] $d=n+2=a_n$;
\item[(2)] $n=5$, $d=6=a_n$.
\end{enumerate} 
If we assume that $x_1\leq \cdots \leq x_{n-1}$ then $(1,\ldots,1,2,1)$ is a fundamental solution in case \textup{(1)}; and $(1,1,1,3,1)$ is a fundamental solution in case \textup{(2)}.
\end{cor}

\begin{rem}
Theorem \ref{T6} gives us the possibility to calculate for a given $n$ all values $a_1, \ldots, a_n, d\in \mathbb N$  with $a_i|d$ for $i=1,\ldots,n$ and $gcd(a_i,a_j)=1$ for $i\ne j$, such that $L_n\ne \emptyset$. Nevertheless this would be a difficult task. We have the complete solution for the case $n=3$.
\end{rem}

\begin{thm}\cite{R2} \label{T7}~\\
Let $a,b,c,d \in \mathbb N$ with $1 \leq a \leq b \leq c$, $a|d$, $b|d$, $c|d$, $gcd(a,b)=gcd(b,c)=gcd(a,c)=1$.\\
The diophantine equation 
\begin{align*}
ax^2+by^2+cz^2 = dxyz
\end{align*}
has a solution $(x,y,z) \in \mathbb N^3$ if and only if one of the following cases holds:

\begin{table}[h]
	%\begin{small}
	\begin{center}
	%\setlength{\tabcolsep}{.3\tabcolsep}
	%\caption{Nielsen transformations (NTs) of the dealer I}\label{vorgehen}
	\begin{tabular}{c | c | c | c | c }
	$a$ &$b$ & $c$ & $d$  & Fundamental solution $(x,y,z)$\\
	\hline
	$1$ & $1$ & $1$ & $1$ & $(3,3,3)$\\
	$1$ & $1$ & $1$ & $3$ & $(1,1,1)$\\
	$1$ & $1$ & $2$ & $4$ & $(1,1,1)$\\
	$1$ & $2$ & $3$ & $6$ & $(1,1,1)$\\
	$1$ & $1$ & $2$ & $2$ & $(2,2,2)$\\
	$1$ & $1$ & $5$ & $5$ & $(1,2,1)$ and $(2,1,1)$
	\end{tabular}
	\end{center}
	%\end{small}
	\end{table}
\end{thm}

\begin{rem}
\begin{enumerate}
\item[1.)] If $x^2+y^2+z^2 = xyz$ for $x,y,z \in \mathbb N$ then 
\begin{align*}
x &\equiv 0 \pmod 3,\\
y &\equiv 0 \pmod 3 \ \text{ and }\\
z &\equiv 0 \pmod 3,
\end{align*}
and division by $9$ leads to the equation 
\begin{align*}
x^2+y^2+z^2 = 3xyz.
\end{align*}
If $x^2+y^2+2z^2 = 2xyz$ for $x,y,z \in \mathbb N$ then
\begin{align*}
x &\equiv 0 \pmod 2,\\
y &\equiv 0 \pmod 2 \ \text{ and }\\
z &\equiv 0 \pmod 2
\end{align*}
and division by $4$ leads to the equation 
\begin{align*}
x^2+y^2+2z^2 = 4xyz.
\end{align*}
In this sense we are left with the four equations
\begin{align*}
x^2+\phantom{2}y^2+\phantom{2}z^2 &= 3xyz,\\
x^2+\phantom{2}y^2+2z^2 &= 4xyz,\\
x^2+2y^2+3z^2 &= 6xyz \ \text{ and }\\
x^2+\phantom{2}y^2+5z^2 &= 5xyz.
\end{align*}
These four equations are in 1-1 correspondence with the four $PGL(2,\mathbb R)$-conjugacy classes of the four free two generator arithmetic Fuchsian groups of genus $1$ (see \cite{FR}, \cite{R2}, \cite{T}).
\item[2.)] Theorem \ref{T6} gives us the possibility to describe the frequency such that $F_n \ne \emptyset$ for any fixed $n \geq 3$ by adapting the argument in \cite{Ba1}.\\
Let $a_1,\ldots, a_n, d \in \mathbb N$ with $a_1 \leq a_2 \leq \cdots \leq a_n$, $a_i | d$ for $i=1,\ldots,n$ and $gcd(a_i,a_j)=1$ for $i\ne j$. Let $\overline{d}=(a_1,\ldots,a_n,d) \in \mathbb N^{n+1}$ and define 
\begin{align*}
A(n):= \# \{ \overline{d} \in \mathbb N^{n+1} | F_n \ne \emptyset \} 
\end{align*}
Then $A(n)= O(n^{\frac{1}{2}+ \epsilon})$ for every $\epsilon >0$.
\end{enumerate} 
\end{rem}

%%%%%%%%%%%%%%%%%%%%%%%%%%%%%%%%%%%%%%%%%%%%%%%%%%%%%%%
%%%%%%%%%%%%%%Section 4%%%%%%%%%%%%%%%%%%%%%%%%%%%%%%%%
%%%%%%%%%%%%%%%%%%%%%%%%%%%%%%%%%%%%%%%%%%%%%%%%%%%%%%%
\section{The Baragar-Umeda equation}\label{sec4}

In this section we consider the diophantine equations \eqref{BU}, which are
\begin{align*}
ax^2+by^2+cz^2 = dxyz+e
\end{align*}
with $a,b,c,d,e \in \mathbb N$ such that $a|d$, $b|d$, $c|d$ and $gdc(a,b,c)=1$. It is obvious that we here assume $gcd(a,b,c)=1$ because if $gcd(a,b,c)=t$ then $t|e$.\\
We write 
\begin{align*}
L:=\{(x,y,z) \in \mathbb Z^3 | ax^2+by^2+cz^2=dxyz + e\}.
\end{align*}
We assume that $L \ne \emptyset$. With $(x,y,z) \in L$ we also have 
\begin{align*}
\left( \frac{d}{a}yz-x,y,z \right) \in L, \quad \left( x, \frac{d}{b}xz-y,z \right) \in L, \quad \text{ and } \quad \left( x,y, \frac{d}{c}xy-z \right) \in L.
\end{align*}
Let $M$ be the group generated by
\begin{align*}
\psi_1: (x,y,z) &\mapsto \left( \frac{d}{a}yz-x,y,z \right),\\
\psi_2: (x,y,z) &\mapsto \left( x,\frac{d}{b}xz-y,z \right) \ \text{ and}\\
\psi_3: (x,y,z) &\mapsto \left( x, y, \frac{d}{c}xy-z \right).
\end{align*}
Again, $M$ is the free product $\mathbb Z_2 * \mathbb Z_2 * \mathbb Z_2$ of three cyclic groups of order $2$ (recall that $\psi_i^2 = 1$  for $i=1,2,3$). This especially shows that $L \cap \mathbb N^3 \ne \emptyset$ if $L \ne \emptyset$.
As suggested by Baragar and Umeda we ask for fundamental solutions of \eqref{BU} in the sense of Lemma \ref{L3}, that is, for solutions $(x,y,z) \in \mathbb N^3$ of \eqref{BU} with 
\begin{align*}
1 \leq x \leq \frac{d}{2a}yz, \quad 1 \leq y \leq \frac{d}{2b}xz \ \text{ and } \ 1 \leq z \leq \frac{d}{2c}xy.
\end{align*}
Let $F$ be the set of fundamental solutions of \eqref{BU}. It is possible that $F = \emptyset$ although if $L \ne \emptyset$ (if $L = \emptyset$ then certainly $F=\emptyset$, also).

\begin{thm}\label{T8}
If $F \ne \emptyset$ then $1 \leq e \leq 4$.
\end{thm}

\begin{proof}
Let $(x,y,z) \in F$ be a fundamental solution of \eqref{BU}. For symmetric reasons we may assume that 
\begin{align*}
1 \leq \sqrt{a} \cdot x \leq \sqrt{b} \cdot y \leq \sqrt{c} \cdot z.
\end{align*}
We define 
\begin{align*}
x':= \sqrt{a} \cdot x, \quad y':=\sqrt{b} \cdot y, \quad z':=\sqrt{c} \cdot z \ \text{ and } \ d':=\frac{d}{\sqrt{abc}},
\end{align*}
and get the equation 
\begin{align*}
x'^2+y'^2+z'^2 = d'x'y'z' + e
\end{align*}
and hence,
\begin{align*}
(d'x'y'-2z')^2=d'^2x'^2y'^2 - 4x'^2 -4y'^2 +4e.
\end{align*}
From $2cz \leq dxy$ we get $2z' \leq d'x'y'$.
Hence
\begin{align*}
y' \leq z' = \frac{d'}{2}x'y'- \frac{1}{2}\sqrt{d'^2x'^2y'^2 - 4x'^2 - 4y'^2+4e}
\end{align*}
and therefore
\begin{align*}
d'x'y'^2\leq 2y'^2+x'^2 - e \quad \text{ and } \quad  y'^2(d'x'-2)\leq x'^2 - e.
\end{align*}
We have two cases to consider. \\

\hspace{-0.8cm}\underline{Case I:} $d'x' - 2 \leq 0$ \ and\\
\underline{Case II:} $d'x' - 2 > 0$.\\

\hspace{-0.8cm}\underline{Case I:} Let $d'x' - 2 \leq 0$, that is, $d'x' \leq 2$.
We know that 
\begin{align*}
d'x'=\frac{d}{\sqrt{abc}}\sqrt{a} \cdot x = \frac{d}{\sqrt{bc}} x.
\end{align*}
Now, since $b|d$ and $c|d$ we have that $\frac{d}{\sqrt{bc}} \geq 1$.
Therefore we have $1 \leq x \leq 2$.

Let first $x=2$. Then necessarily $d'=\frac{d}{\sqrt{bc}}=1$ because $\frac{d}{\sqrt{bc}}\cdot 2 \geq 2$ and $d'x'=\frac{d}{\sqrt{bc}} x \leq 2$. Therefore $a=1$ (and $bc=d^2$, that is, $b=c=d$), and hence $0 \leq 4-e$, that is, $e \leq 4$, because now $d'x'=2$.

Now let $x=1$.
Then $1 \leq d'x' = \frac{d}{\sqrt{bc}} < 2$. Let $c \geq b$ (the case $b \geq c$ is analogously). Then $1 \leq \frac{d}{c} < 2$ because $c| d$ and necessarily $d=c$.\\
But then $1 \leq \sqrt{\frac{d}{b}} < 2$ and $\frac{d}{b} \in \mathbb N$. Therefore $d=b$, $d=2b$ or $d=3b$.
Now, $(x,y,z)$ is a fundamental solution of \eqref{BU}, that is here,
\begin{align*}
1 \leq z \leq \frac{d}{2c}xy=\frac{y}{2} \quad \text{ and }\quad 1 \leq y \leq \frac{d}{2b}xz \leq \frac{d}{2b} \cdot \frac{y}{2}
\end{align*} 
which gives $4b \leq d$, a contradiction. Hence, $x=1$ cannot occur in case I.\\

\hspace{-0.8cm}\underline{Case II:} Let $d'x' - 2 > 0$. Then 
\begin{align*}
0 < y'^2 (d'x'-2) \leq x'^2 - e \quad \text{ and } \quad y'^2(d'x'-3) \leq -e.
\end{align*}
Hence, $2 < d'x' = \frac{d}{\sqrt{bc}}x < 3$ and $5 \leq \frac{d^2}{bc} x^2 \leq 8$ because $\frac{d^2}{bc} \in \mathbb N$.\\
Therefor $x=2$ or $x=1$.
\begin{enumerate}
\item Let first $x=2$.\\
Since $5 \leq \frac{d^2}{bc}x^2 \leq 8$ we must have $\frac{d^2}{bc}=2$. Then 
\begin{align*}
b=d, \ d=2c \qquad  \text{ or } \qquad c=d, \ d=2b
\end{align*}
and therefore $a=1$ or $a=2$ because $gcd(a,b,c)=1$. In both cases $e \leq 7$ because $0 < x'^2 - e$.
\item Now, let $x=1$. Then 
\begin{align*}
5 \leq \frac{d^2}{bc} = \frac{d}{b} \cdot \frac{d}{c} \leq 8,
\end{align*}
and hence
\begin{align*}
\frac{d^2}{bc}= 5, 6, 7 \text{ or } 8.
\end{align*}
Then $a=1,2,3,4,5,6,7$ or $8$ because $gcd(a,b)=1$.\\
This can be seen as follows.\\
Let $p$ be a prime divisor of $a$. Then $p|d$, and, hence, $p|\frac{d}{b}$ or $p|\frac{d}{c}$ because $gcd(a,b,c)=1$.  So $a|\frac{d^2}{bc}$. In all cases $e \leq 7$ because
\begin{align*}
0 < x'^2 - e = a-e.
\end{align*}
\end{enumerate}
For the next part of the proof we remark that 
\begin{align*}
ax^2 = x'^2 \leq 8
\end{align*}
because $a|\frac{d^2}{bc}$ and $\frac{d^2}{bc}x^2 \leq 8$.\\

So far we have $1\leq e \leq 7$.\\

To prove the theorem 
we first show that $e \leq 6$.\\
Assume that $e=7$. Then $x'^2= ax^2=8$, that is,
\begin{align*}
x=2, \ a=2, \ \frac{d^2}{bc}=2 \quad \text{ or } \quad x=1, \ a=8, \ \frac{d^2}{bc}=8
\end{align*}
because $5 \leq \frac{d^2}{bc} x\leq 8$ and $a | \frac{d^2}{bc}$.\\
We first consider the case $x=1$, $a=\frac{d^2}{bc}=8$.\\
From 
\begin{align*}
y'^2 (d'x'-2) \leq x'^2 -7 =1
\end{align*}
we get
\begin{align*}
by^2 \leq \frac{1}{2 \sqrt{2}-2} < 2,
\end{align*}
 and hence, $y=1$ and $b=1$. This contradicts that $\sqrt{a} \cdot x \leq \sqrt{b} \cdot y$.\\
 Now, let 
\begin{align*}
x=a=\frac{d^2}{bc}=2.
\end{align*}
Again we have
\begin{align*}
x'^2=8 \quad \text{ and } \quad  d'x' = \frac{d}{\sqrt{bc}}x=2 \sqrt{2} \quad \text{ and } \quad y'^2=by^2 < 2.
\end{align*}
Then again $y=1$ and $b=1$ and this contradicts $\sqrt{a} \cdot x \leq \sqrt{b} \cdot y$. \\
Hence we have $e \leq 6$.\\

We now show that $e \leq 5$.\\
Assume that $e=6$. Then $x'^2=ax^2=7$ or $8$.
Let first be $x'^2=8$. Then again 
\begin{align*}
x=2=a=\frac{d^2}{bc} \quad \text{ or } \quad x=1, \ a=\frac{d^2}{bc}=8.
\end{align*}
In both cases
\begin{align*}
by^2 = y'^2 \leq \frac{2}{2 \sqrt{2}-2} < 4,
\end{align*}
that is, $by^2\leq 3$. Then $y=1$, $1 \leq b \leq 3$. In all cases this contradicts $\sqrt{a} \cdot x \leq \sqrt{b} \cdot y$.\\
Now, let $x'^2=ax^2=7$.\\
Then $x=1$, $a=7=\frac{d^2}{bc}$. Hence, $d'x'=\sqrt{7}$. Then
\begin{align*}
1 \leq by^2 = y'^2 \leq \frac{1}{\sqrt{7} - 2} < \frac{5}{3},
\end{align*}
that is $b=y=1$, and again this contradicts $\sqrt{a} \cdot x \leq \sqrt{b} \cdot y$.\\
This gives $e  \leq 5$.\\

We now show that $e \leq 4$.\\
Assume that $e=5$. Then $x'^2=ax^2 = 6$, $7$ or $8$. If $x'^2=8$ we get
\begin{align*}
1 \leq by^2 = y'^2 \leq \frac{3}{2 \sqrt{2} -2} < 6
\end{align*}
which gives 
\begin{align*}
y=1, \ 1\leq b \leq 5 \quad \text{ or } \quad y=2, \ b=1.
\end{align*}
In both cases this contradicts $\sqrt{a} \cdot x = 2 \sqrt{2} \leq \sqrt{b} \cdot y$.\\
If $x'^2 =ax^2 = 7$ we get $1 \leq by^2 = y'^2 < \frac{10}{3}$ that is $y=1$, $1 \leq b \leq 3$.
This again contradicts $\sqrt{a} \cdot x = \sqrt{7} \leq \sqrt{b} \cdot y$.\\
Finally, let $x'^2=ax^2=6$.\\
Then we must have $x=1$ and $a=6$. Since $a|\frac{d^2}{bc}$ we also get $\frac{d^2}{bc}=6$.
 Therefore $d'x'= \sqrt{6}$ and 
\begin{align*}
y'^2 = by^2 \leq \frac{1}{\sqrt{6}-2} < 3,
\end{align*}
 that is $y=1$, $1 \leq b \leq 2$. Again this contradicts $\sqrt{a} \cdot x = \sqrt{6} \leq \sqrt{b}y$. This gives $e \leq 4$ and proves Theorem \ref{T8}.

\end{proof}

\begin{rem}
\begin{enumerate}

\item We now give a very simple, direct  and different proof of Theorem \ref{T8} for the case that in addition $d=\sqrt{abc}$, especially for the case $a=1$, $b=c=d$.\\
Let 
\begin{align*}
ax^2 + by^2 + cz^2 = dxyz + e
\end{align*}
with $a,b,c,d,e \in \mathbb N$ such that $a|d$, $b|d$, $c|d$ and $gcd(a,b,c)=1$; and let in addition $d=\sqrt{abc}$. Write $x':=\sqrt{a} \cdot x$, $y':=\sqrt{b} \cdot y$ and $z':=\sqrt{c} \cdot z$. Then 
\begin{align*}
x'^2 + y'^2 + z'^2 = x'y'z' + e.
\end{align*}
Let $F \ne \emptyset$ and assume that $5 \leq e$.\\
 Then we may construct $A,B \in SL(2, \mathbb R)$ with $tr A = x'$, $tr B = y'$, $tr AB= z'$ and $tr ABA^{-1}B^{-1}=x'^2+y'^2+z'^2 -x'y'z'-2 \geq 3$. The algorithmic method, developed in \cite{K-IR}, now automatically gives a contradiction. Hence, $e\leq 4$.

\item Let 
\begin{align*}
ax^2+by^2+cz^2=dxyz + e
\end{align*} 
with $a, b, c, d,e \in \mathbb N$ such that $a|d$, $b|d$, $c|d$, $gcd(a,b,c)=1$ and $1 \leq e \leq 3$.
The following list, see Table \ref{diphantinsolution}, gives all diophantine equations for which a fundamental solution exist, and in each case we give the fundamental solutions. To start the list, without loss of generality, we may assume that $1 \leq a \leq b \leq c$. We remark that the list for $e=1$ is due to Baragar and Umeda \cite{BU}. 

\newpage

\begin{table}[ht]
	%\begin{small}
	\begin{center}
	\caption{Diophantine equations with fundamental solutions}\label{diphantinsolution}
	\begin{tabular}{c | c }
	Equation & Fundamental solutions \\
	\hline
	& \\
	$\phantom{1}x^2+\phantom{1}5y^2+\phantom{1}5z^2=\phantom{1}5xyz+1$ & $(4,1,2)$ and $(4,2,1)$ \\
	& \\
	$\phantom{1}x^2+\phantom{1}3y^2+\phantom{1}6z^2=\phantom{1}6xyz+1$ & $(2,1,1)$ \\
	& \\
	$3x^2+\phantom{1}4y^2+\phantom{1}6z^2=12xyz+1$ & $(1,1,1)$ \\
	& \\
	$2x^2+\phantom{1}7y^2+14z^2=14xyz+1$ & $(2,1,1)$ \\
	& \\
	$2x^2+\phantom{1}2y^2+\phantom{1}3z^2=\phantom{1}6xyz+1$ & $(1,1,1)$ \\
	& \\
	$6x^2+10y^2+15z^2=30xyz+1$ & $(1,1,1)$ \\
	& \\
	$\phantom{1}x^2+\phantom{1}2y^2+\phantom{1}2z^2=\phantom{1}2xyz+1$ & $(3,2,2)$ \\
	& \\
	$\phantom{1}x^2+\phantom{1}8y^2+\phantom{1}8z^2=\phantom{1}8xyz+1$ & $(3,1,1)$ \\
	& \\
	\hline
	& \\
	$3x^2+5y^2+15z^2=15xyz+2$ & $(2,1,1)$ \\
	& \\
	$2x^2+3y^2+\phantom{1}6z^2=\phantom{1}6xyz+2$ & $(2,2,1)$ \\
	& \\
	$\phantom{1}x^2+7y^2+\phantom{1}7z^2=\phantom{1}7xyz+2$ & $(3,1,1)$ \\
	& \\
	\hline
	& \\
	$2x^2+5y^2+10z^2=10xyz+3$ & $(2,1,1)$ \\
	& \\
	$\phantom{1}x^2+6y^2+\phantom{1}6z^2=\phantom{1}6xyz+3$ & $(3,1,1)$ \\
	\end{tabular}
	\end{center}
	%\end{small}
\end{table}
If $3x^2+4y^2+6z^2=12xyz+1$ then the coefficient $b=4$ is not squarefree, and if we replace $y$ by $y'=2y$ then we get the equation $3x^2+y'^2+6z^2=6xy'z+1$, and in this sense these two equations are equivalent.\\ 
If $x^2+8y^2+8z^2=8xyz+1$ then the coefficients $b=8$ and $c=8$ are not squarefree, and if we replace $y$ by $y'=2y$ and $z$ by $z'=2z$ then we get the equation $x^2+2y'^2+2z'^2=2xy'z'+1$, and in this sense these two equations are equivalent.
\item Now let
\begin{align*}
ax^2+by^2+cz^2 = dxyz + 4
\end{align*}
with $a,b,c,d \in \mathbb N$ such that $a|d$, $b|d$, $c|d$ and $gcd(a,b,c)=1$.\\
Let $a \leq b \leq c$. The diophantine equation has a fundamental solution if and only if $a=1$ and $b=c=d$, that is, 
\begin{align*}
x^2+dy^2+dz^2 = dxyz + 4,
\end{align*}
and this equation has infinitely many fundamental solutions $(2,n,n)$ with $\sqrt{d}\cdot n \geq 2$. For $1 \leq d \leq 3$ then $(2,1,1)$ is not a fundamental solution.
\item  Let
\begin{align*}
ax^2+by^2+cz^2 = dxyz + e
\end{align*}
with $a,b,c,d,e \in \mathbb N$ such that $a|d$, $b|d$, $c|d$ and $gcd(a,b,c)=1$. Let in addition $d=\sqrt{abc}$ and $1 \leq e \leq 2$.\\
 Assume further that $F \ne \emptyset$ and $(x,y,z) \in F$.
Then we may construct a group $G= \ <A,B> \ \subset PSL(2, \mathbb R)$ with $tr A=\sqrt{a} \cdot x$, $tr B=\sqrt{b} \cdot y$ and $tr AB =\sqrt{c} \cdot z$. Then $G$ is a discret subgroup of $PSL(2, \mathbb R)$ with a presentation 
\begin{align*}
\langle  A,B | (ABA^{-1}B^{-1})^n = 1\rangle
\end{align*}
where $n=3$ if $e=1$ and $n=2$ if $e=2$, that is, $G$ has signature $(1;n)$ where $n=3$ if $e=1$ and $n=2$ if $e=2$ (see \cite{FR}). In fact, in both cases $G$ is an arithmetic  Fuchsian group with invariant trace field $\mathbb Q$ (see \cite{T}).
\item Let
\begin{align*}
ax^2+by^2+cz^2 = dxyz + 3
\end{align*}
with $a,b,c,d \in \mathbb N$ such that $a|d$, $b|d$, $c|d$ and $gcd(a,b,c)=1$.\\ Assume further that $F \ne \emptyset$ and $(x,y,z) \in F$. Recall from the above list (Table~\ref{diphantinsolution})  that  here $d=\sqrt{abc}$. We may construct a group $G= \ <A,B>$ with $G \subset PSL(2, \mathbb R)$, $tr A=\sqrt{a} \cdot x$, $tr B=\sqrt{b} \cdot y$ and $tr AB =\sqrt{c} \cdot z$. Then $G$ is a discrete subgroup of $PSL(2, \mathbb R)$ with a presentation
\begin{align*}
\langle s_1, s_2, s_3 | s_1^2=s_2^2=s_3^2=(s_1s_2s_3)^3=1 \rangle,
\end{align*}
where $A=s_1s_2$, $B=s_3s_1$, that is, G has signature $(0;2,2,2,3)$ (see~\cite{FR}). In fact, $G$ is an arithmetic Fuchsian group with invariant trace field $\mathbb Q$ (see \cite{MR} and \cite{ANR}).
\item Altogether, let 
\begin{align*}
ax^2+by^2+cz^2 = dxyz + e
\end{align*}
with $a,b,c,d,e \in \mathbb N$ such that $a \leq b \leq c$, and $a,b,c$ squarefree,  $a|d$, $b|d$, $c|d$, $gcd(a,b,c)=1$, $d= \sqrt{abc}$ and $1 \leq e \leq 3$. From the list (Table~\ref{diphantinsolution}) we see that there are exactly nine such diophantine equations. These nine equations are in 1-1 correspondence with the nine $PGL(2,\mathbb R)$-conjugacy classes of the two generator arithmetic Fuchsian groups of a signature $(1;2)$, $(1;3)$ or $(0;2,2,2,3)$ and with invariant trace field $\mathbb Q$ (see \cite{MR}, \cite{ANR} and \cite{T}).
\item In fact, we could in general assume that $a,b,c$ are squarefree, without loss of generality.
This can be seen as follows. Let, for instance, $a=p^2a'$ with $p>1$. 
Then we define $x':=px$, $d':=\frac{d}{p}$ and get $$a'x'^2+by^2+cz^2=d'x'yz.$$
If $p$ $\not|$ $b$ and $p$ $\not|$ $c$ then $p$ $|$ $d'$ and $c$ $|$ $d'$.\\
Let, for instance, $p$ $|$ $b$. Then $p$ $\not|$ $c$ because $gcd(a,b,c)=1$. If $p^2$ $\not|$ $b$ then we still have $b$ $|$ $d'$ because $p^2$ $|$ $d$; certainly $c$ $|$ $d'$ and $gcd(a',b,c)=1$. If $p^2$ $|$ $b$, that is, $b=p^2b'$ then, in addition we define $y':=py$ and $d'':=\frac{d'}{p}=\frac{d}{p^2}$ and get the equation $$a'x'^2+b'y'^2+cz^2=d''x'y'z$$ with $a'$ $|$ $d''$, $b'$ $|$ $d''$, $c$ $|$ $d''$ and $gcd(a',b',c)=1$.
\end{enumerate}
\end{rem}

%%%%%%%%%%%%%%%%%%%%%%%%%%%%%%%%%%%%%%%%%%%%%
%%%%%%%%%%%%%% Biblithek %%%%%%%%%%%%%%%%%%%%
%%%%%%%%%%%%%%%%%%%%%%%%%%%%%%%%%%%%%%%%%%%%%

%%%%%%%%%%%%%%%%%%%%%%%%%%%%%%%%%%%%%%%%%%%%%%%%%%%%%%%
%%%%%%%%%%%%%%References%%%%%%%%%%%%%%%%%%%%%%%%%%%%%%%
%%%%%%%%%%%%%%%%%%%%%%%%%%%%%%%%%%%%%%%%%%%%%%%%%%%%%%%
%%% ----------------------------------------------------------------------

%%%%%%%%%%%%%%%%%%%%%%%%%%%%%%%%%%%%%%%%%%%%%%%%%%%%%
%%%%%%%%%%%%% Author information %%%%%%%%%%%%%%%%%%%%
%%%%%%%%%%%%%%%%%%%%%%%%%%%%%%%%%%%%%%%%%%%%%%%%%%%%%

\section*{Author information}
Benjamin Fine, Department of Mathematics, Fairfield University\\
Fairfield, Connecticut 06430, USA.\\
E-mail:\texttt{fine@fairfield.edu}\\

Gabriele Kern-Isberner, Fakultät für Informatik, TU Dortmund,\\
 Otto-Hahn-Strasse 12, 44227 Dortmund, Germany.\\
E-mail: \texttt{gabriele.kern-isberner@cs.uni-dortmund.de}\\

Anja I. S. Moldenhauer, Fachbereich Mathematik, Universität Hamburg,\\
Bundesstrasse 55, 20146 Hamburg, Germany.\\
E-mail: \texttt{anja.moldenhauer@uni-hamburg.de}\\

Gerhard Rosenberger, Fachbereich Mathematik, Universität Hamburg,\\
Bundesstrasse 55, 20146 Hamburg, Germany.\\
E-mail: \texttt{gerhard.rosenberger@math.uni-hamburg.de}\\


\begin{thebibliography}{1234}
\bibitem{ANR} P. Ackermann, M. N\"a\"at\"anen and G. Rosenberger, \textit{The arithmetic Fuchsian groups with signature $(0;2,2,2,q)$.} Recent Advances in Group Theory and Low-Dimensional Topology (ed. J. R. Cho and J. Mennicke); Heldermann-Verlag (2003), 51--64.

\bibitem{Ba1} A. Baragar, \textit{Integral solutions of Markoff-Hurwitz equations.} J. Number Theory \textbf{49} (1994), 27--44.

\bibitem{Ba2} A. Baragar, \textit{Asymptotic growth of Markoff-Hurwitz numbers.} Compositio Math. \textbf{94} (1994), 1--18.

\bibitem{Ba3} A. Baragar, \textit{The Markoff-Hurwitz equation over number fields.} Rocky Mountain J. Math. \textbf{35} (2005), 695--712.

\bibitem{Ba4} A. Baragar, \textit{Products of consecutive integers and the Markoff equation.} Aequationes Math.  \textbf{51} (1996), 129--136.

\bibitem{Ba5} A. Baragar, \textit{The Hurwitz Equations.} Number Theory with an Emphasis on the Markoff Spectrum (ed. A. Pollington and W. Moran); Lecture Notes in Pure and Applied Math. \textbf{147}, Marcel Dekker (1993), 1--8.

\bibitem{Ba6} A. Baragar, \textit{The Markoff Equation and Equations of Hurwitz.} Thesis, Brown University (1991).

\bibitem{BU} A. Baragar and K. Umeda, \textit{The asymptotic growth of integer solutions to the Rosenberger equation.} Bull. Austral. Math. Soc. \textbf{69} (2004), 481--497.

\bibitem{B} Yu. N. Baulina, \textit{Fundamental solutions of the equation $x_1^2+ \cdots +x_n^2=mx_1 \cdots x_n$.} Mat. Zometki \textbf{52} (1992), 136--137.


\bibitem{BLS} A. F. Beardon, J. Lehner and M. Sheingorn, \textit{Closed geodesics on a Riemann Surface
with application to the Markov spectrum.} Trans. AMS \textbf{295} (1986), 635--647.

\bibitem{BBH} A. Beineke, T. Bruestle and L. Hille and an appendix by O. Kerner, \textit{Cluster-Cyclic Quivers with three vertices and the Markov Equation.} Algebr. Represent. Theory \textbf{14} (2011), 97--112.


\bibitem{BMR} B. H. Bowditch, C. Maclachlan and A. W. Reid, \textit{Arithmetic hyperbolic surface bundles.} Math. Ann. \textbf{302} (1995), 31--60.

\bibitem{BR} C. Baer and G. Rosenberger, \textit{The equation $ax^2+by^2+cz^2=dxyz$ over quadratic imaginary fields.} Result. Math. \textbf{33} (1998), 30--39.


\bibitem{C1} H. Cohn, \textit{Markoff Forms and Primitive words.} Math. Ann.  \textbf{196} (1972), 8--22.

\bibitem{C2} H. Cohn, \textit{Markoff geodesics in matrix theory.} Number Theory with an Emphasis on the Markoff Spectrum (ed. A. D. Pollington and W. Moran); Lecture Notes in Pure an Applied Math. \textbf{147}, Marcel Dekker (1993), 69--82.

\bibitem{FR} B. Fine and G. Rosenberger, \textit{Classification of all generating pairs of two generator Fuchsian groups.} Proc. of the Galway/St. Andrews Conf. on Groups 1993, London Math. Soc. Lecture Note Ser. \textbf{211} (1995), 205--232.

\bibitem{G} W. M. Goldman, \textit{The modular group action on real $SL(2)$-characters of a one-holed torus.} Geometry \& Topology \textbf{7} (2003), 443--486.

\bibitem{G-JT} E. Gonz\'{a}lez-Jim\'{e}nez and J. M. Tornero, \textit{Markoff-Rosenberger triples in arithmetic progression.} J. Symbolic Computation \textbf{53} (2013), 53--63.


\bibitem{He} N. P. Herzberg, \textit{On a problem of Hurwitz.} Pac. J. Math. \textbf{50} (1974), 485--493.

\bibitem{Hi} F. Hirzebruch, \textit{The signature theorem: Reminiscences and recreation.} Manuskript (1970).

\bibitem{H} A. Hurwitz, \textit{\"Uber eine Aufgabe der unbestimmten Analysis.} Archiv. Math. Phys. \textbf{3} (1907), 185--196.

\bibitem{HM} D. Holt and M. Macbeath, \textit{Certain maximal characteristic subgroups of free groups of rank 2.} Comm. in Algebra \textbf{25} (1997), 1047--1077.

\bibitem{KR} R. N. Kalia and G. Rosenberger, \textit{Automorphisms of the Fuchsian groups of type $(0;2,2,2,q;0)$.} Comm. Algebra \textbf{6} (1978), 1115--1129.

\bibitem{K} G. Kern, \textit{Die Gleichung $a_1x_1^2+ \cdots +a_nx_n^2=dx_1 \cdots x_n - k$.} Diplomarbeit, Universit\"at Dortmund (1979).

\bibitem{K-IR} G. Kern-Isberner and G. Rosenberger, \textit{\"Uber Diskretheitsbedingungen und die diophatische Gleichung $ax^2+by^2+cz^2=dxyz$.} Arch. Math. \textbf{34} (1980), 481--493.

\bibitem{LS} J. Lehner and M. Sheingorn, \textit{Simple closed geodesics on $H^3 / \Gamma(3)$ arise from the Markov spectrum.} Bull. AMS \textbf{11} (1984), 359--362.

\bibitem{MR} C. Maclachlan and G. Rosenberger, \textit{Two-generator arithmetic Fuchsian groups II.} Math. Proc. Comb. Phil. Soc. \textbf{111} (1992), 7--24.

\bibitem{M} A. A. Markoff, \textit{Sur les formes quadratiques binaires ind\'{e}finies.} Math. Ann. \textbf{17} (1880), 379--399.

\bibitem{Mo1} L. J. Mordell, \textit{On the integer solutions of the equation $x^2+y^2+z^2+2xyz=n$.} J. London Math. Soc. \textbf{28} (1953), 500-510.

\bibitem{Mo2} L. J. Mordell, \textit{Diophantine Equations.} Academic Press (1969).

\bibitem{P} D. Panagopoulos, \textit{A secret sharing scheme using groups} arXiv:1009.0026v1 (2010).

\bibitem{R1} G. Rosenberger, \textit{Zu Fragen der Analysis im Zusammenhang mit der Glei\-chung $x_1^2+\cdots+x_n^2-ax_1 \cdots x_n=b$.} Mh. Math. \textbf{85} (1977), 211--233.

\bibitem{R2} G. Rosenberger, \textit{\"Uber die diophantische Gleichung $ax^2+by^2+cz^2=dxyz$.} J. Reine Angew. Math. \textbf{305} (1979), 122--125.

\bibitem{R3} G. Rosenberger, \textit{\"Uber Tschebyscheff Polynome,
 Nicht--Kongruenzuntergrup\-pen der Modulgruppe und Fibonacci-Zahlen.} Math. Ann. \textbf{246} (1980), 193--203.

\bibitem{R4} G. Rosenberger, \textit{Fuchssche Gruppen, die freies Produkt zweier zyklischer Gruppen sind und die Gleichung $x^2+y^2+z^2=xyz$.} Math. Ann. \textbf{199} (1972), 213--227.

\bibitem{R5} G. Rosenberger, \textit{A note on Fibonacci and related numbers in the theory of $2 \times 2$ matrices.} Fibonacci numbers and their applications (ed. A. N. Philippon et all); D. Reidel Publishing Company (1986), 235--240.

\bibitem{S1} J. H. Silverman, \textit{The arithmetic of elliptic curves.} Springer Verlag (1986).

\bibitem{S2} J. H. Silverman, \textit{The Markoff equation $x^2+y^2+z^2=axyz$ over quadratic imaginary fields.} J. Number Theory \textbf{35} (1990), 72--104.

\bibitem{Sch} A. L. Schmidt, \textit{Minimum of quadratic forms with respect to Fuchsian groups I and II.} J. Reine Angew. Math. \textbf{286/287} (1976), 341--368 and \textbf{292} (1977), 109--114.

\bibitem{Sh1} M. Sheingorn, \textit{Rational solutions of $\sum_{i=1}^{3} a_ix_i^2=dx_1x_2x_3$.} Holomorphic Functions and Moduli I (ed. D. Drasin, C. F. Earle, F. W. Gehring, I. Kra, A. Marden); Springer-Verlag (1988), 229--236.

\bibitem{Sh2} M. Sheingorn, \textit{Characterization of simple closed geodesics on Fricke surfaces.} Duke Math. J. \textbf{52} (1985), 535--545.

\bibitem{T} K. Takeuchi, \textit{Arithmetic Fuchsian Groups with signature $(1;e)$.} J. Math. Soc. Japan \textbf{35} (1983), 381--407.

\bibitem{W} L. Wang, \textit{Rational points and canonical heights on $K3$-surfaces in $\mathbb P^1~\times~\mathbb P^1~\times~\mathbb P^1$.} Contemporary Math. \textbf{186} (1995), 273--289.

\bibitem{Wo1} A. J. Woldar, \textit{On Hurwitz generation and genus actions of sporadic groups.} Illinois J. Math. \textbf{33} (1989), 416--437.

\bibitem{Wo2} A. J. Woldar, \textit{Genus actions of finite simple groups.} Illinois J. Math. \textbf{33} (1989), 438--450.

\bibitem{Z} D. Zagier, \textit{On the number of Markoff numbers below a given bound.} Math. Comp. \textbf{39} (1982), 709--723.

% \bibitem{test} A. B. C. Test, \textit{On a Test.} J. of Testing
% \textbf{88} (2000), 100--120.
% \bibitem{latex} G. Gr\"atzer, \textit{Math into \LaTeX.} 3rd Edition,
% Birkh\"auser, 2000.
\end{thebibliography}
\end{document}